\documentclass[11pt]{amsart}

\usepackage{amsfonts}
\usepackage{amsthm}
\usepackage{amssymb}
\usepackage{amsmath}
\usepackage{url, hyperref}

\newtheorem{theo}{\bf Theorem}[section]
\newtheorem{lemma}{\bf Lemma}[section]
\newtheorem{coro}{\bf Corollary}[section]
\newtheorem{rem}{\bf Remark}[section]

\newcommand{\D}{{\mathcal D}}

\newcommand{\F}{{\mathcal F}}

\newcommand{\spn}{{\rm span}}

\newcommand{\Hp}{{\mathcal H}}
\newcommand{\Z}{{\mathbb Z}}

\newcommand{\Q}{{\Bbb Q}}
\newcommand{\R}{{\mathbb R}}

\newcommand{\bea}{\begin{eqnarray*}}
\newcommand{\eea}{\end{eqnarray*}}
\newcommand{\be}{\begin{eqnarray}}
\newcommand{\ee}{\end{eqnarray}}

\newcommand{\ve}{\boldsymbol}

\newcommand{\prob}{\mbox{\rm Prob}\,}

\numberwithin{equation}{section}

\begin{document}

\title[Average Behavior of the Frobenius Numbers]{Integer Knapsacks: \\ Average Behavior of the Frobenius Numbers}
\author{Iskander Aliev}
\address{School of Mathematics and Wales Institute of Mathematical and Computational Sciences, Cardiff University, Senghennydd Road, CARDIFF, Wales, UK}
\email{alievi@cf.ac.uk}

\author{Martin Henk}
\address{Institut f\"ur Algebra und Geometrie, Otto-von-Guericke
Universit\"at Mag\-deburg, Universit\"atsplatz 2, D-39106-Magdeburg,
Germany} \email{henk@math.uni-magdeburg.de}

%\subjclass[2000]{????}
%\keywords{Knapsack problem, Frobenius number, successive minima, inhomogeneous minimum, distribution of lattices}

\begin{abstract}
Given a primitive integer vector ${\ve a}\in\Z^N_{>0}$,  the largest integer $b$ such that the knapsack polytope
$P=\{{\ve x}\in \R^N_{\ge 0}: \langle{\ve a}, {\ve x}\rangle=b\}$
contains no integer point is called the Frobenius number of ${\ve a}$. We show that the asymptotic growth of the Frobenius number in average is
significantly slower than the growth of the maximum Frobenius number. More precisely, we prove that it does not essentially exceed $||{\ve a}||_\infty^{1+1/(N-1)}$,
where $||\cdot||_\infty$ denotes the maximum norm.
\end{abstract}

\maketitle

\section{Introduction and statement of results}

% Let $a_1, a_2, \ldots, a_N$ be positive integers with $\gcd(a_1, a_2,\ldots, a_N)=1$. %and $a_i \le b$, $1\le i \le N$.
%
For a positive integral vector ${\ve a}=(a_1, a_2, \ldots, a_N)\in\Z^N_{>0}$ with $\gcd({\ve a})=\gcd(a_1, a_2,\ldots, a_N)=1$  and a positive integer $b$ the {\em knapsack polytope} $P=P({\ve a},b)$ is defined as
\bea
P=\{{\ve x}\in \R^N_{\ge 0}: \langle{\ve a}, {\ve x}\rangle=b\}\,,
\label{P}
\eea
where $\langle\cdot, \cdot\rangle$ denotes the inner product. The
integer programming feasibility problem:
\be
\mbox{Does the polytope}\; P\; \mbox{contain an integer vector?}
\label{Knapsack}
\ee
is called the  {\em integer knapsack problem} and is well-known to be NP-complete (cf., e.g., Karp \cite{Karp}).

Given the input vector ${\ve a}\in\Z^N$,  the largest integral value $b$ such that the instance of (\ref{Knapsack}) is
infeasible is called the {\em Frobenius number} of ${\ve a}$\,, denoted by $g_N=g_N({\ve a})$.
The Frobenius number plays an important role in the analysis of integer programming algorithms (see, e.g.,  Aardal and Lenstra \cite{Aardal_Lenstra}, Hansen and Ryan \cite{Hansen_Ryan}, and Lee, Onn and Weismantel \cite{Lee_Onn_Weismantel}) and, vice versa, integer programming algorithms are known to be an effective tool for computing the Frobenius number (see Beihoffer et al \cite{BHNW}). The general problem of finding $g_N$ has been traditionally refereed to as the {\em Frobenius problem}. There is a rich literature on the various aspects of this question. For an impressive list of references see Ramirez Alfonsin \cite{Alf}.

Computing $g_N$ when $N$ is not fixed is an NP-hard problem (Ramirez Alfonsin \cite{Alf1}). For any fixed $N$ the Frobenius number $g_N$ can be found in polynomial time by a sophisticated algorithm due to Kannan \cite{Kannan}. One should mention here that, due to its complexity, Kannan's algorithm has apparently never been implemented.

From the viewpoint of analysis of integer programming algorithms, upper bounds on the Frobenius number $g_N({\ve a})$ in terms of the input vector ${\ve a}$ are of primary interest. Known results include classical upper
bounds by Erd\H{o}s and Graham \cite{EG}
\be g_N\le 2 a_N\left [\frac{a_1}{N}\right ]-a_1\,, \label{EGb}\ee
by Selmer \cite{Se}
\be g_N\le 2 a_{N-1}\left [\frac{a_N}{N}\right ]-a_N\,, \label{Sb}\ee
by Vitek \cite{Vi}
\be g_N\le \left [\frac{(a_2-1)(a_N-2)}{2}\right ]-1\, \label{Vb}\ee
 and by many other authors,
as well as more recent results by Beck, Diaz, and Robins \cite{BDR}
\be g_N\le \frac{1}{2}\left(\sqrt{a_1 a_2
a_3(a_1+a_2+a_3)}-a_1-a_2-a_3\right)\,, \label{BDRb}\ee
(assuming in (\ref{EGb})--(\ref{BDRb}) $a_1\le a_2\le \ldots\le a_N$) and by Fukshansky and Robins \cite{Lenny}, who produced an upper
bound
in terms of the covering radius of a lattice related to the integers
$a_1, \ldots, a_N$.

In the most interesting case $a_i\sim a_j$, $i,j=1,\ldots,N$,  all known upper bounds are of order $||{\ve a}||_\infty^2$, where $||\cdot||_\infty$ denotes the maximum norm. This is especially transparent
in the case of the results (\ref{EGb})--(\ref{BDRb}). For $N=3$ Beck and Zacks \cite{BZ} conjectured that, except of a special family of input vectors,
the Frobenius number does not exceed $C(a_1 a_2 a_3)^{\alpha}$ with absolute constants $C$ and $\alpha<2/3$. This conjecture has been disproved by
 Schlage-Puchta \cite{SP}. As a special case, the latter result implies that, roughly speaking, cutting off special families of input vectors cannot make the order of upper bounds smaller than $||{\ve a}||_{\infty}^2$.
In general, one can show that the quantity $||{\ve a}||_{\infty}^2$ plays a role of a limit for estimating the Frobenius number $g_N$ from above.

The next natural and important question is to derive a good upper estimate for the Frobenius number of a ``typical'' input vector ${\ve a}$. This problem appears to be hard, and to the best of our knowledge it has firstly been systematically investigated by V.~I.~Arnold, see, e.g., \cite{Arnold99, Arnold06, Arnold07}. In particular, he conjectured that $g_N(\bf a)$ grows like $T^{1+1/(N-1)}$ for a ``typical'' $\bf a$ of $1$-norm $T$.
Recently,  Bourgain and  Sinai \cite{BS} proved a statement in the spirit of that conjecture, which says, roughly speaking, that
\begin{equation}
 \prob_{\infty,\alpha}\left(g_N({\bf a})/T^{1+1/(N-1)}\geq D \right) \leq \epsilon(D),
\label{eq:BS}
\end{equation}
where $\prob_{\infty,\alpha}(\cdot)$ is meant with respect to the uniform
distrubition among all points in the set
\begin{equation*}
   G_{\infty,\alpha}(N,T)=\{{\bf a}\in\Z^N_{>0}: \gcd({\ve a})=1,\,\Vert {\bf a}\Vert_\infty\leq T,\,a_i> \alpha\,T,\,1\leq i\leq N\},
\end{equation*}
 where $0<\alpha<1$ is a fixed number. The number $\epsilon(D)$ does not depend on $T$ and tends to zero as $D$ approaches infinity.
Our main result below also implies that \eqref{eq:BS} (see Corollary \ref{main_coro})  holds for the more general and natural  case $\alpha=0$.

In order to state our main theorem, we have to fix some further notation. Put $\Sigma ({\ve a})=\sum_{i=1}^N
a_i$ and $\Pi({\ve a})=(\prod_{i=1}^N a_i)^{1/(N-1)}$.
Theorem 2.5 of Kannan \cite{Kannan} indicates that, from the geometric viewpoint, it is more convenient to study the quantity
\bea f_N({\ve a})=g_N({\ve a})+ \Sigma ({\ve a})\,.\eea
Clearly, $f_N=f_N({\ve a})$ is
the largest integer which is not a {\em positive} integer
combination of $a_1, \ldots, a_N$.
In this paper we study the  asymptotic behavior of the ratio $f_N({\ve a})/s({\ve a})$ with
\bea
s({\ve a})=\frac{\sum_{i=1}^N ||{\ve a}[i]||a_i}{||{\ve a}||^{1-1/(N-1)}}\,,
\eea
 where ${\ve a}[i]=(a_1,\ldots,a_{i-1},a_{i+1},\ldots,a_N)$ and $||\cdot||$ denotes the Euclidean norm. The geometric meaning of the normalization $s({\ve a})^{-1}$ is explained in Section \ref{FL}. With these notation let
%In this paper we will be interested in the asymptotic behavior of the ``average'' Frobenius number $g_N({\ve a})$ as the euclidean norm of ${\ve a}$ tends to infinity.
%
\bea
G(N, T)= \{{\ve a}\in \Z^N_{> 0}: \gcd({\ve a})=1\,, ||{\ve a}||\le T\},
\eea
and let $\prob_{N,T}(\cdot)$ be the uniform probability distribution on $G(N, T)$.
The main result of the paper is
\begin{theo}
For $N\ge 3$ the inequality
\bea
\prob_{N,T}(f_N({\ve a})/s({\ve a})>t)\ll_N t^{-2}\,
\eea
holds. Here $\ll_N$ denotes the Vinogradov symbol with the constant depending on $N$ only.
\label{main}
\end{theo}
In terms of $g_N$, $T$ and $\prob_{\infty,0}(\cdot)$ we obtain the following corollary.
\begin{coro}
For $N\ge 3$ the inequality
\bea
\prob_{\infty,0}(g_N({\ve a})/T^{1+1/(N-1)}>t)\ll_N t^{-2}\,
\eea
holds.
\label{main_coro}
\end{coro}

Beihoffer et al \cite{BHNW} performed extensive computations which lead to a conjecture that $\Pi({\ve a})$ is a good predictor for the average value of
$f_N({\ve a})$. Indeed, they conjectured that the average value of $f_N({\ve a})/\Pi({\ve a})$ is asymptotically equal to a small constant. An analogous conjecture for $N=3$ was proposed in Davison \cite{Dav}.

One should remark here that $\Pi({\ve a})$ is essentially a lower bound for $f_N$.
The main result of Aliev and Gruber \cite{Aliev-Gruber} states that the inhomogeneous minimum
$\mu_0=\mu_0(S_{N-1})$ of the standard simplex
\bea S_{N-1}=\{(x_1,\ldots,x_{N-1})\in \R_{\ge 0}^{N-1}:
\sum_{i=1}^{N-1}x_i \le 1\}\,\eea
is a sharp lower bound for the ratio $f_N({\ve a})/\Pi({\ve a})$.

The next theorem answers a question similar to the conjecture of Beihoffer et al with respect to a different  normalization of $f_N$.
\begin{theo}
For $N\ge 3$ we have
\bea
\sup_{T}\frac{\sum_{{\ve a}\in G(N,T)}f_N({\ve a})/s({\ve a})}{\# G(N,T)}\ll_N 1\,.
\eea
\label{Asymptotic_bound_1}
\end{theo}

Observe that for all ${\ve a}$ we have $s({\ve a})\ll_N||{\ve a}||_\infty^{1+1/(N-1)}$. This implies the following result.
\begin{coro}
For $N\ge 3$ we have
\be
\sup_{T}\frac{\sum_{{\ve a}\in G(N,T)}f_N({\ve a})/||{\ve a}||_\infty^{1+1/(N-1)}}{\# G(N,T)}\ll_N 1\,.
\label{average_via_maximum}
\ee
\label{Awesome}
\end{coro}

Obviously, the maximum norm $||{\ve a}||_\infty$ in (\ref{average_via_maximum}) can be replaced by any other norm.  Moreover, applying arguments similar to the one given  in the proof of Corollary \ref{main_coro}, one can also replace the Euclidean norm in the definition of $G(N, T)$ by any other norm, which, for example,  for the maximum norm leads to the set  $G_{\infty,0}(N,T)$.

Corollary \ref{Awesome} says that
 the asymptotic growth of the Frobenius number in average is
significantly slower than the growth of the maximum Frobenius number as $||{\ve a}||$ tends to infinity. Moreover, perhaps surprisingly, the average Frobenius number, as $N\rightarrow\infty$, does not essentially exceed $||{\ve a}||_\infty$.

The next result shows that the ratio $f_N({\ve a})/s({\ve a})$ is unbounded along any given ``direction'' ${\ve
\alpha}\in \R^N$,  so that Theorem \ref{Asymptotic_bound_1} is not straightforward.

\begin{theo} For any $\epsilon>0$, $M>0$ and for any ${\ve
\alpha}=(\alpha_1,\alpha_2,\ldots,\alpha_{N-1},1)$, $0\le\alpha_1\le\alpha_2\le\ldots\le\alpha_{N-1}\le 1$, there exists
an integer vector ${\ve a}=(a_1,a_2,\ldots, a_N)$ with  $0<a_1<a_2< \ldots< a_N$ and $\gcd({\ve a})=1$  such that
\be || {\ve \alpha}-\frac{1}{a_N}{\ve a}||_\infty<\epsilon
\,\label{Density2}\ee
and
\be \frac{f_N({\ve a})}{s({\ve a})}> M \,. \label{Sharpness2}\ee

\label{only_asymptotic}
\end{theo}

The paper is organized as follows. In Section 2 we combine Kannan's formula for $f_N({\ve a})$ with Jarnik's inequalities in order to reformulate the problem via Minkowski's successive minima.  Section 3 is devoted to Schmidt's results on the distribution of sublattices of $\Z^n$ on which our work heavily relies. For the proof of Theorem \ref{only_asymptotic} we need a density lemma which will be presented in Section 4. In the subsequent sections we give the proofs of our main results.

\section{Frobenius number and lattices}
\label{FL}
Following the geometric approach developed in Kannan \cite{Kannan} and Kannan and Lovasz
\cite{Kannan-Lovasz}, we will make use of tools from the geometry of
numbers.

By lattice we understand a discrete submodule $L$ of a finite--dimensional Euclidean space.
Recall that a
family of sets in
$\R^{N-1}$ is a {\em covering} if their union equals $\R^{N-1}$.
Given a %ray
set $S$ and a lattice $L$, we say that $L$ is a {\em covering
lattice} for
$S$ %we denote by $(S,L)$ the
if the family $\{S+{\ve l}: {\ve l}\in L\}$
%
%If $(S,L)$
is a covering. %, then $L$ is a {\em covering lattice} for $S$.
The {\em inhomogeneous minimum} of %the ray
the set $S$ with respect to
%a
the lattice $L$ is the quantity
\bea \mu(S,L)= \inf\{\sigma>0: L \,\,\text{is a covering lattice
of}\,\, \sigma S\} \eea
and the quantity
\bea \mu_0(S)=\inf\{\mu(S,L): \det L =1\}\, \eea
is called the {\em (absolute) inhomogeneous minimum} of $S$. If $S$ is
bounded and has inner points,
%is Lebesgue measurable and $\vol(S)<\infty$
then $\mu_0(S)$ does not vanish and is finite (see Gruber and Lekkerkerker \cite{GrLek},
Chapter 3).
The quantity $\mu_0(S)$ is closely related to the, perhaps better known, {\em covering
constant} $\Gamma(S)$ of the set $S$, where $\Gamma(S)= \sup\{\det(L): L \,\,\text{is a covering lattice of}\,\,
S\}$.
Indeed, by Gruber and Lekkerkerker \cite[p.~230]{GrLek} we have
$\mu_0(S)=\Gamma(S)^{-1/(N-1)}$.

Depending on the vector ${\ve a}\in\Z^N_{>0}$ we define the following $S_{\ve a}$ and lattice  $L_{\ve a}$ by
\begin{equation*}
\begin{split}
S_{\ve a}  &=\left\{ {\ve x} \in\R_{\ge 0}^{N-1}: \,
\sum_{i=1}^{N-1}a_i\, x_i \le 1\right\},\\
L_{\ve a}  &=\left\{{\ve x}\in \Z^{N-1}:\,
\sum_{i=1}^{N-1}a_i x_i \equiv 0 \mod a_N\right\}.
\end{split}
\end{equation*}
Kannan \cite[Theorem 2.5]{Kannan} proved that
\be f_N({\ve a})=\mu(S_{\ve a}, L_{\ve a}),
\label{result_of_Kannan}\ee
which provides a starting point for geometric investigations of the Frobenius number.
To this end we
 define the hyperplane lattice $\Lambda_{\ve a}(t)$ in $\R^N$ as
\bea
\Lambda_{\ve a}(t)=\{{\ve x}\in \Z^N: \langle {\ve a}, {\ve x}\rangle = t\}\,.
\eea
Let $V_{\ve a}(t)={\rm aff}\, \Lambda_{\ve a}(t)$ and $S_{\ve a}(t)$ be the $(N-1)$--dimensional simplex $V_{\ve a}(t)\cap \R_{\ge 0}^N$.
For convenience we will also use the notation $V_{\ve a}=V_{\ve a}(0)$ and $ \Lambda_{\ve a}= \Lambda_{\ve a}(0)$.
%Let us also observe that the simplex $S_{\ve a}(t)$ has the vertices ${\ve v}_i=(t/a_i){\ve e}_i$, $i=1,\ldots, N$, where ${\ve e}_1,\ldots, {\ve %e}_N$
%are the standard basis vectors.

Furthermore, let $\pi(\cdot)$ denote the orthogonal projection onto coordinate hyperplane corresponding to the variables $x_1,\ldots,x_{N-1}$.
Then clearly $S_{\ve a}=\pi(S_{\ve a}(1))$, $L_{\ve a}=\pi(\Lambda_{\ve a}(0))$ and, since inhomogeneous minima are independent with respect to regular
affine transformations,
we can write \eqref{result_of_Kannan} as
\be f_N({\ve a})=\mu(S_{\ve a}(1), \Lambda_{\ve a}(1))\,.
\label{projection_of_Kannan}\ee
Here and through the rest of the paper we consider $V_{\ve a}(t)$ as a usual $(N-1)$--dimensional Euclidean space.

By a standard calculation (see, e.g., Fukshansky and Robins \cite[(19)]{Lenny}) the inradius of the simplex $S_{\ve a}(t)$ is given by
\bea
r_{\ve a}(t)=\frac{t||{\ve a}||}{\sum_{i=1}^N ||{\ve a}[i]||a_i}\,.
\eea
Denoting by  $B_r^N$  the ball of radius $r$ in $\R^N$ we have
 by (\ref{projection_of_Kannan})
\bea
f_N({\ve a})\le \mu(B^N_{r_{\ve a}(1)}\cap V_{\ve a}, \Lambda_{\ve a}).
\eea
Observe that %the inhomogeneous minimum $\mu(S, L)$ satisfies
$\mu(S, tL)=t\mu(S, L)$ and $\mu(tS, L)=t^{-1}\mu(S,
L)$.
Thus
\be
f_N({\ve a})\le \frac{||{\ve a}||^{1/(N-1)}}{r_{\ve a}(1)}\mu(B^N_{1}\cap V_{\ve a}, \Gamma_{\ve a})\,,
\label{f_N_mu}
\ee
where the lattice $\Gamma_{\ve a}=||{\ve a}||^{-1/(N-1)}\Lambda_{\ve a}$ has determinant $1$. In order to estimate $\mu(B^N_{1}\cap V_{\ve a}, \Gamma_{\ve a})$ we need Minkowski's successive minima, which for a $o$-symmetric convex set $K$ and a lattice $\Lambda$ defined by (see \cite[pp.~375]{peterbible})
\begin{equation*}
   \lambda_i(K,\Lambda)=\inf\{\lambda>0 : \dim(\lambda\,K\cap\Lambda)\geq i\},\,\quad  1\leq i\leq \dim\Lambda.
\end{equation*}

Let $\lambda_i=\lambda_i(B^N_{1}\cap V_{\ve a}, \Gamma_{\ve a})$ be the $i$-th successive minimum of the ball $B^N_{1}\cap V_{\ve a}$
with respect to the lattice $\Gamma_{\ve a}$.

By Jarnik's inequalities (see, e.g., Gruber and Lekkerkerker \cite[pp.~99]{GrLek}), we have
\be
\frac{1}{2}\lambda_{N-1}\le \mu(B^N_{1}\cap V_{\ve a}, \Gamma_{\ve a})\le \frac{N-1}{2} \lambda_{N-1}\,.
\label{mu_lambda}
\ee
Thus, for a fixed dimension $N$ the inhomogeneous minimum is essentially equal to the last successive minimum. By (\ref{f_N_mu}) and the right--hand side of (\ref{mu_lambda}) we obtain the inequality
\be
\frac{2f_N({\ve a})}{(N-1)\,s({\ve a})}=\frac{2r_{\ve a}(1)}{||{\ve a}||^{1/(N-1)}}f_N({\ve a})\le\lambda_{N-1}\,.
\label{f_N_lambda}
\ee
The latter expression explains the geometric meaning of the quantity $s({\ve a})^{-1}$. This is the normalized radius of a ball inscribed into the simplex $S_{\ve a}(1)$.
%Applying more delicate isoperimetric inequalities one can slightly improve the normalization, but we do not go further in this
%direction.

\section{Distribution of sublattices of $\Z^m$}

In this section we will recall several results due to W. Schmidt \cite{Schmidt} on the distribution of integer lattices.
Two lattices $L$, $L'$
are similar if there is a linear bijection $\phi: L\rightarrow L'$ such that for some fixed $c>0$ we have $||\phi({\ve x})||=c||{\ve x}||$.
Let ${\tilde O}_n$ be the group of matrices $K=({\ve k}_1, \ldots, {\ve k}_n)\in GL_n(\R)$ whose columns ${\ve k}_1, \ldots, {\ve k}_n$
have $||{\ve k}_1||= \cdots= ||{\ve k}_n||\neq 0$ and inner products $\langle {\ve k}_i, {\ve k}_j \rangle=0$ for $i\neq j$.
When $X=({\ve x}_1, \ldots, {\ve x}_n)\in GL_n(\R)$, we may uniquely write the matrix $X$ in the form
\be
X=KZ\,,
\label{Schmidt1.2}
\ee
where $K\in {{\tilde O}_n}$ and
\be
Z=\left ( \begin{array}{llll} 1 & x_{12} & \cdots & x_{1n}\\
                              0 & y_2 & \cdots & x_{2n}\\
                              \vdots\\
                              0 & 0 & \cdots & y_n
                         \end{array}\right )
\label{H_matrix}\ee
with $y_2,\ldots,y_n>0$. The matrices $Z$ as in (\ref{H_matrix}) form the generalized upper half--plane $\Hp=\Hp_n$. For $Z\in\Hp$ and
$M\in GL_n(\R)$, we may write $ZM$ in the form (\ref{Schmidt1.2}), that is we uniquely have $ZM=KZ_M$ with $K\in{\tilde O}_n$ and $Z_M\in \Hp$.
Thus $GL_n(\R)$ acts on $\Hp$; to $M$ corresponds the map $Z\mapsto Z_M$. In particular, $GL_n(\Z)$, as a subgroup of $GL_n(\R)$, acts on $\Hp$.
We will denote by $\F$ a fundamental domain for the action of $GL_n(\Z)$ on $\Hp$. We will also write $\mu$
for the $GL_n(\R)$ invariant measure on $\Hp$ with $\mu(\F)=1$.

%Let $L$ be a lattice of rank $n$ in $\R^m$ with basis $X=({\ve x}_1, \ldots, {\ve x}_n)$
%The lattices $L$, $L'$ are similar precisely if they have the same image

Suppose now that $1<n\le m$. There is a map (see p. 38 of Schmidt \cite{Schmidt} for detail) from lattices of rank $n$ in $\R^m$ onto the set $\Hp/GL_n(\Z)$ of orbits of $GL_n(\Z)$ in $\Hp$. The lattices $L$, $L'$ are similar precisely if they have the same image in $\Hp/GL_n(\Z)$, hence the same image in $\F$. Similarity classes of lattices are parametrized by the elements of a fundamental domain $\F$.

A subset $\D\subset \Hp$ is called {\em lean} if $\D$ is contained in some fundamental domain $\F$. For $a>0$, $b>0$, let $\Hp(a,b)$
consists of $Z\in\Hp$ (in the form (\ref{H_matrix})) with
\bea
y_{i+1}\ge a y_i\,,\;\;\;1\le i <n, \quad
|x_{ij}|\le b y_i\,,\;\;\;1\le i<j\le n.
\eea
Here we assume $y_1=1$.

Recall that the Frobenius number $g_N({\ve a})$ is well-defined only for integer vectors ${\ve a}=(a_1, a_2, \ldots,a_N)$ with $\gcd(a_1, a_2, \ldots, a_N)=1$. The vectors ${\ve a}$ with this property are called {\em primitive}.
More generally, a lattice $L\subset \Z^m$ is {\em primitive} if $L=\spn_{\R}(L)\cap \Z^m$.
Clearly, there is one-to-one correspondence between primitive vectors ${\ve b}\in\Z^n$ and the primitive $(n-1)$--dimensional sublattices
 $\Lambda_{\ve b}$. Note also that $\det \Lambda_{\ve b}=||{\ve b}||$.
Let $P(\D, T)$, where $\D$ is lean, be the number of primitive lattices $L\subset\Z^m$ with similarity class in $\D$ and determinant $\le T$.
\begin{theo}[Schmidt \protect{\cite[Theorem 2]{Schmidt}}]
Suppose $1<n<m$ and let $\D\subset\Hp(a,b)$ be lean and Jordan-measurable. Then, as $T\rightarrow\infty$,
\be
P(\D, T)\sim c_2(m,n) \mu(\D)T^m
\label{Schmidt_1.9}
\ee
with
\bea
c_2(m,n)=\frac{1}{m} {m \choose n} \frac{V_{m-n+1}\cdots V_m}{V_1 V_2 \cdots V_n}\cdot \frac{\zeta(2)\cdots\zeta(n)}{\zeta(m-n+1)\cdots\zeta(m)}\,.
\eea
Here $V_l$ is the volume of the unit ball in $\R^l$ and $\zeta(\cdot)$ is the Riemann zeta--function.
\label{Schmidt_th_2}
\end{theo}
Thus, roughly speaking, the proportion of primitive lattices with similarity class in $\D$ is $\mu(\D)$.

Given a vector ${\ve u}=(u_1, u_2, \ldots, u_n)$ with $u_i\ge 1 \;(1\le i <n)$, the lattices $L$ with
\bea
\frac{\lambda_{i+1}(B_1^n\cap \spn_{\R}(L),L)}{\lambda_i(B_1^n\cap \spn_{\R}(L), L)}\ge u_i
\eea
form a set of similarity classes, which will be denoted by
$\D({\ve u})$.
\begin{theo}[Schmidt \protect{\cite[Theorme 5 (i)]{Schmidt}}]
The set $\D({\ve u})$ may be realized as a lean, Jordan--measurable subset of $\Hp$. We have
\be
\mu(\D({\ve u}))\ll_{m,n}\prod_{i=1}^{n-1} u_i^{-i(n-i)}\,.
\label{Th_5_i}
\ee
Here $\ll_{m,n}$ denotes the Vinogradov symbol with the constant depending on $m$ and $n$ only.
\label{Schmidt_th_5}
\end{theo}
\begin{rem}
Note that the condition $\D({\ve u})\subset \Hp(a,b)$ of Theorem \ref{Schmidt_th_2} is also satisfied for some constants
$a=a(n)$ and $b=b(n)$. We refer the reader \cite[p.~58]{Schmidt} for further detail.
\end{rem}

\section{A density lemma}

\begin{lemma}
Let $L$ be a lattice with basis ${\ve b}_1,\ldots,{\ve b}_{N-1}$,
${\ve b}_i\in\mathbb Q^{N}$, $1\le i \le N-1$, and let ${\ve \alpha}=(\alpha_1,\ldots, \alpha_{N-1},1)\in\Q^{N}$ be a vector orthogonal to $L$.
%and for all
%rationals $\alpha_1, \ldots,\alpha_{N-1}$ with
%$0<\alpha_1\le\alpha_2\le\cdots\le\alpha_{N-1}\le 1$,
Then there exists
%an infinite arithmetic progression ${\mathcal P}$ and
a sequence ${\ve a}(t)=(a_1(t),\ldots,a_{N-1}(t),a_{N}(t))\in\Z^{N}$, $t=1,2,\ldots$,
such that $\gcd({\ve a}(t))=1$ and the following properties hold:
\begin{itemize}
\item[(i)]
The lattice $\Lambda_{{\ve
a}(t)}$ has a basis ${\ve b}_1(t),\ldots,{\ve b}_{N-1}(t)$
with
\be \frac{b_{ij}(t)}{d\,t}=b_{ij}+O\left(\frac{1}{t}\right)\,,\;\;\;
(1\le i,j\le {N})\,, \label{asympt_aij} \ee
where $d\in\mathbb N$ is such that $d\, b_{ij}\in\mathbb Z$, $1\le i\le N-1$ and $1\le j\le {N}$.
\item[(ii)] The last component of the vector ${\ve a}(t)$ satisfies
\be a_{N}(t)=\det(\pi(L))d^{N-1}t^{N-1}+O(t^{N-2}).\label{asympt_h} \ee
\item[(iii)] The sequence $(a_{N}(t))^{-1}{\ve
a}(t)$ converges to ${\ve \alpha}$. Indeed, \be \frac{a_i(t)}{a_{N}(t)}
=\alpha_i+O\left(\frac{1}{t}\right), 1\le i\le N-1. \label{asympt_alpha}\ee

%Moreover,
%
%\be a_{N}(t)=\det(L)d^{N-1}t^{N-1}+O(t^{N-2})\label{asympt_h} \ee
%
%and
%
%\be \alpha_i(t):=\frac{a_i(t)}{a_{N}(t)}
%=\alpha_i+O\left(\frac{1}{t}\right). \label{asympt_alpha}\ee
%
\end{itemize}
\label{main_lemma}
\end{lemma}

The result is a modified version of Theorem 1.2 of Aliev and Gruber \cite{Aliev-Gruber}, but in order to keep the paper self-contained as much as possible we give a short proof here.

\begin{proof}
%\label{Proof_of_the_main_lemma}

%Let ${\ve u}=(u_1, \ldots, u_{N-1}, -u_N)$ be the outer product of the vectors ${\ve b}_1, \ldots, {\ve b}_{N-1}$
%and let
%
%\bea
%\alpha_1=\frac{u_1}{u_N}\,,\alpha_2=\frac{u_2}{u_N}\,,\ldots, \alpha_{N-1}=\frac{u_{N-1}}{u_N}\,.
%\eea
%

Let us consider the matrices
\bea B=\left (
\begin{array}{ccccc}
b_{11} & b_{12}  & \ldots & b_{1\,N-1} &
b_{1\,N}\\
b_{21} & b_{22}  & \ldots & b_{2\,N-1} &
b_{2\,N}\\
\vdots & \vdots & & \vdots & \vdots\\
b_{N-1\,1} & b_{N-1\,2}  & \ldots & b_{N-1\,N-1} &
b_{N-1\,N}\\
\end{array}
\right ) \eea
and
\bea M=M(t,t_1,\ldots,t_{N-1})\eea\bea =\left (
\begin{array}{ccccc}
db_{11}t+t_1 & db_{12}t  & \ldots & db_{1\,N-1}t &
db_{1\,N}t\\
db_{21}t & db_{22}t+t_2  & \ldots & db_{2\,N-1}t &
db_{2\,N}t\\
\vdots & \vdots & & \vdots & \vdots\\
db_{N-1\,1}t & db_{N-1\,2}t  & \ldots & db_{N-1\,N-1}t+t_{N-1} &
db_{N-1\,N}t\\
\end{array}
\right ) \,,\eea
where $t, t_1, \ldots, t_{N-1}$ are variables.
%
%Denote by $M_i=M_i(t,t_1,\ldots,t_{N-1})$  the minor
%obtained by omitting the $i$th column in $M$.
%
%Following the proof of Theorem 2 in \cite{NewSL}, we observe that
% $M_1,\ldots,M_{N}$ have no non--constant common factor.

Denote by $M_i=M_i(t,t_1,\ldots,t_{N-1})$ and $B_i$ the minors
obtained by omitting the $i$th column in $M$ or in $B$,
respectively.
Note that
\begin{equation}
\begin{split}
|B_{N}| & = |\det(b_{ij})|=\det(\pi(L)),\quad \alpha_i=\frac{|B_i|}{|B_{N}|},\text{ and } \\
M_i&=d^{N-1}B_it^{N-1}+ \text{polynomials in $t$ of degree less than $N-1$}.
\end{split}
\label{asympt_minors}
\end{equation}
Following the proof of Theorem 2 in Schinzel \cite{NewSL} we also observe that
$M_1,\ldots,M_{N}$ have no non--constant common factor. By \cite[Theorem 1]{NewSL} with $m=1$, $F=1$, and
$F_{1\nu}=M_\nu(t,t_1,\ldots,t_{N-1})$, $1\le \nu\le N$,
%Here we consider $t, t_1, \ldots, t_k$ as variables.
there exist integers $t^*_1,\ldots,t^*_{N-1}$ and an infinite
arithmetic progression ${\mathcal P}$ such that for $at+b\in{\mathcal
P}$
\bea\mbox{gcd}(M_1(at+b,t^*_1,\ldots,t^*_{N-1}), \ldots
,M_{N}(at+b,t^*_1,\ldots,t^*_{N-1}))=1 \,.\eea
For $t=1,2,\ldots$ we set
\bea {\ve a}(t)=(M_1(at+b,t^*_1,\ldots,t^*_{N-1}), \ldots
,(-1)^{N-1}M_{N}(at+b,t^*_1,\ldots,t^*_{N-1})).\eea
Then the basis ${\ve b}_1(t),\ldots,{\ve b}_{N-1}(t)$ for $L_{{\ve
a}(t)}$ satisfying the statement of Lemma \ref{main_lemma} is
given by the rows of the matrix  $M(t,t^*_1,\ldots,t^*_{N-1})$.
The properties \eqref{asympt_minors} of minors $M_i$,
$B_i$ imply the properties (\ref{asympt_aij})--(\ref{asympt_alpha})
of the sequence ${\ve a}(t)$.

\end{proof}

\section{Proof of Theorem \ref{main}}

We consider the sequence of discrete random variables
$X_T: G(N,T)\rightarrow \R_{\ge 0}$ defined as
\bea
X_T({\ve a})= \frac{f_N({\ve a})}{s({\ve a})}\,.
\eea
Recall that the {\em cumulative distribution function} (CDF) $F_T$ of $X_T$ is defined for $t\in\R_{\ge 0}$ as
\bea
F_T(t)=\prob_{N,T}(X_T\le t\,)\,.
\eea

For a real number $u\ge 1$, let ${\ve v}_i(u)=(u_1, u_2, \ldots, u_{N-2})$ be the vector with $u_i=u$ and $u_j=1$ for all $j\neq i$. Define the set $\D(u)$ of similarity classes as (cf.~Section 3)
\bea
\D(u)= \bigcup_{i=1}^{N-2}\D({\ve v}_i(u))\,.
\eea
By (\ref{Th_5_i}) the measure of this set satisfies
\be
\mu(\D(u))\ll_N \frac{1}{{u}^{N-2}}\,.
\label{measure_of_union}
\ee
%
%Here and in the rest of the proof the constants in the Vonogradov symbols depend on $N$ only.
%
Let $Y_T: G(N,T)\rightarrow \R_{>0}$ be the sequence of random variables defined as
\bea
Y_T({\ve a})= \sup\{v\in \R_{> 0}: \Lambda_{\ve a}\in \D(c_1v^{2/(N-2)})\}\,,
\eea
where the constant $c_1=c_1(N)$ is given by
\bea
c_1= {V_{N-1}^{\frac{2}{(N-1)(N-2)}}}/{(N-1)^{2/(N-2)}}\,.
\eea
Since the set $\D(1)$ contains all similarity classes we have  for all ${\ve a}\in G(N,T)$
\be Y_T({\ve a})\ge c_1^{-(N-2)/2}.\label{l_bound}\ee
Let now $\Gamma\subset\R^N$ be a lattice of rank $N-1$ and determinant $1$, and  let $\lambda_i:=\lambda_{i}(B_1^N\cap \spn_{\R}(L),L)$, $1\le i\le N-1$. We need the following simple observation
\begin{lemma}
Let
$\lambda_{N-1}>\lambda>0$. Then there exists an index $i\in\{1,\dots,N-2\}$ with
\bea
\frac{\lambda_{i+1}}{\lambda_i}> c_2(N) \lambda^{2/(N-2)}\,,
\eea
where
$c_2(N)=2^{-\frac{2}{N-2}}V_{N-1}^{\frac{2}{(N-1)(N-2)}}$.
\label{lambda}
\end{lemma}
\begin{proof}
Suppose the opposite, i.e.,
\begin{equation*}
\frac{\lambda_{i+1}}{\lambda_i}\le c_2(N) \lambda^{2/(N-2)},
\end{equation*}
for all $1\le i\le N-2$. Then, $\lambda_{N-1}\leq (c_2(N)\lambda^{2/(N-2)})^{N-1-i}\lambda_{i}$, and
by Minkowski's second fundamental theorem (cf., e.g., \cite[pp.~376]{peterbible})
\be
\lambda_1 \lambda_2 \cdots \lambda_{N-1} \le \frac{2^{N-1}}{V_{N-1}}\,,
\label{2nd_MT}
\ee
we get the contradiction
\bea
\lambda_{N-1}\le {(c_2(N) \lambda^{2/(N-2)})}^{\frac{(N-2)}{2}}  \frac{2}{V_{N-1}^{1/(N-1)}}=\lambda\,.
\eea
\end{proof}
We remark, that  \eqref{2nd_MT} can be slightly improved by applying
Minkowski's second theorem for balls. However we do not go further in this direction.

Let now $\tilde F_T$ be the CDF of the random variable $Y_T$.
\begin{lemma}
For any $T\ge 1$ and $t\ge 0$ we have
\bea
{\tilde F_T}(t)\le F_T(t).
\eea
\label{cdfs}
\end{lemma}
\begin{proof}
Let $\Gamma=\Gamma_{\ve a}$. By (\ref{f_N_lambda}), we have
\bea
\frac{f_N({\ve a})}{s(\ve a)}\le\frac{(N-1)}{2} \lambda_{N-1} \,.
\eea
Hence, if for some $t$ holds
\bea
X_T({\ve a})=\frac{f_N({\ve a})}{s({\ve a})}> t
\eea
then clearly $\lambda_{N-1}>\frac{2t}{(N-1)}$.
By Lemma \ref{lambda}, applied with $\lambda=\frac{2t}{(N-1)}$, we get
\bea
\frac{\lambda_{i+1}}{\lambda_i}>c_1(N) t^{2/(N-2)}\,.
\eea
Consequently, the lattice $\Gamma_{\ve a}$ belongs to a similarity class in $\D(c_1t^{2/(N-2)})$, so that $Y_T({\ve a})> t$.
Therefore,
\bea
\begin{split}
\prob_{N,T}(X_T\le t\,) & =1 - \frac{\#\{{\ve a}\in G(N,T): f_N({\ve a})/s({\ve a})>t\}}{\#G(N,T)} \\
&\ge 1- \frac{\#\{{\ve a}\in G(N,T): Y_T({\ve a})> t\}}{\#G(N,T)}=\prob(Y_T\le t\,).
\end{split}
\eea
\end{proof}

By Schmidt \cite[Theorem2]{Schmidt}, the number of primitive integer vectors ${\ve a}\in \Z^N$  with  $||{\ve a}||\le T$  and which lie on coordinate hyperplanes is essentially equal to $T^{N-1}$, so that the proportion of such vectors tends to zero as $T\rightarrow\infty$.
Thus by Lemma \ref{cdfs} and  Theorem \ref{Schmidt_th_2} we finally obtain:
\bea
\begin{split}
\prob_{N,T}(f_N({\ve a})/s({\ve a})>t) & = 1- F_T(t) \le 1-{\tilde F}_T(t) \\
& =\frac{\#\{{\ve a}\in G(N,T): Y_T({\ve a})> t\}}{\#G(N,T)}
\\ &\ll_N \mu(\D(c_1t^\frac{2}{N-2}))\ll_N t^{-2}.
\end{split}
\eea
This proves the theorem.

%q({\ve a})=\frac{(N-1)||{\ve a}||^{1/(N-1)}}{2r(1)}=\frac{(N-1)^2\sum_{i=1}^N ||{\ve a}[i]||a_i}{2||{\ve a}||^{1-1/(N-1)}}

%\frac{P(\D, T)}{}< c_2(N,N-1) \frac{c_3(N-1)^2}{2C} T^m\,,

\section{Proof of Corollary \ref{main_coro}}

Observe that for all ${\ve a}$ holds $f_N({\ve a})> g_N({\ve a})$. Therefore, it is enough to prove the inequality
\bea
\prob_{\infty,0}(f_N({\ve a})/T^{1+1/(N-1)}>t)\ll_N t^{-2}\,.
\eea
By \cite[Theorem 2]{Schmidt}, we have $\# G(N, T/\sqrt{N})\gg_N \# G(N, T)$ and thus
\bea
\begin{split}
\prob_{\infty,0}(f_N({\ve a})/T^{1+1/(N-1)}>t)&\ll_N \frac{\#\{{\ve a}\in G(N, T): f_N({\ve a})/T^{1+1/(N-1)}>t\}}{\# G(N, T/\sqrt{N})}\\
&\ll_N \prob_{N,T}(f_N({\ve a})/T^{1+1/(N-1)}>t)\,.
\end{split}
\eea
Noting that $s({\ve a})\ll_N T^{1+1/(N-1)}$ for ${\ve a}\in G(N,T)$, we get
\bea
\prob_{N,T}(f_N({\ve a})/T^{1+1/(N-1)}>t)\le \prob_{N,T}(f_N({\ve a})/s({\ve a})>\delta_N t)\,
\eea
with some positive constant $\delta_N$ which depends on $N$ only. Finally, by Theorem \ref{main}, we obtain
the desired inequality:
\bea
\prob_{\infty,0}(f_N({\ve a})/T^{1+1/(N-1)}>t)\ll_N \prob_{N,T}(f_N({\ve a})/s({\ve a})>\delta_N t)\ll_N t^{-2}\,.
\eea
% \bea
% \ll_N t^{-2}\,.
% \eea
%

\section{Proof of Theorem \ref{Asymptotic_bound_1}}

We will keep the notation from the proof of Theorem \ref{main}. Let also $E(\cdot)$ denote the mathematical expectation.
Since for any nonnegative real-valued random variable $X$
\be
E(X)=\int_0^\infty (1-F_X(t))dt\,,
\label{E_F}
\ee
Lemma \ref{cdfs} implies that $E(X_T)\le E(Y_T)$
and, consequently,
\be
\sup_{T}E(X_T)\le \sup_{T}E(Y_T)\,.
\label{E_limsup}
\ee
Next,  by Theorem \ref{Schmidt_th_2} we also have
\bea
1-{\tilde F}_T(t)=\frac{\#\{{\ve a}\in G(N,T): Y_T({\ve a})> t\}}{\#G(N,T)}
\ll_N \mu(\D(c_1t^\frac{2}{N-2}))\ll_N t^{-2}.
\eea
Thus by (\ref{E_F}), (\ref{E_limsup}) and observation (\ref{l_bound}), we obtain
\bea
\sup_{T}E(X_T)\ll_N \int_{c_1^{-(N-2)/2}}^\infty t^{-2}\, dt \ll_N 1,
\eea
which  proves the theorem.

%q({\ve a})=\frac{(N-1)||{\ve a}||^{1/(N-1)}}{2r(1)}=\frac{(N-1)^2\sum_{i=1}^N ||{\ve a}[i]||a_i}{2||{\ve a}||^{1-1/(N-1)}}

%\frac{P(\D, T)}{}< c_2(N,N-1) \frac{c_3(N-1)^2}{2C} T^m\,,

\section{Proof of Theorem \ref{only_asymptotic}}

The proof is based on Lemma \ref{main_lemma} and the following
continuity property of the inhomogeneous minima which follows from a more general result of Gruber \cite[Satz 1]{Peter}. We say that a
sequence $S_t$ of {\em star bodies} in $\R^{N-1}$ converges to a
star body $S$ if the sequence of {\em distance functions} of $S_t$
converges uniformly on the unit ball in $\R^{N-1}$ to the distance
function of $S$. For the notions of star bodies, distance functions and  convergence of a sequence of
lattices to a given lattice we refer the reader to Gruber--Lekkerkerker \cite{GrLek}.
\begin{lemma}[Gruber \protect{\cite[Satz 1]{Peter}}]
Let $S_t$ be a sequence of star bodies in $\mathbb R^{N-1}$ which
converges to a bounded star body $S$ and let $L_t$ be a sequence of
lattices in $\mathbb R^{N-1}$ convergent to a lattice $L$. Then
\bea \lim_{t\rightarrow\infty}\mu(S_t,L_t)=\mu(S,L)\,. \eea
\label{limit}
\end{lemma}
For the proof of  Theorem \ref{only_asymptotic} we may assume
that ${\ve \alpha}\in \Q^{N}$ and
\be 0<\alpha_1<\alpha_2<\ldots <\alpha_{N-1}<
1\,.\label{conditions_on_alpha} \ee
The simplex
\bea S_{\ve \alpha}(1)=\{(x_1,\ldots,x_{N})\in\R_{\ge 0}^N: \sum_{i=1}^{N-1}\alpha_i x_i + x_N= 1\}
\eea
contains a ball of radius
\bea
r_{\ve \alpha}(1)=\frac{||{\ve \alpha}||}{\sum_{i=1}^N ||{\ve \alpha}[i]||\alpha_i}\,.
\eea
Let now $R_{\ve\alpha}$ be the radius of a ball containing $S_{\ve \alpha}(1)$, and let
$c({\ve \alpha})= r_{\ve\alpha}(1)/R_{\ve \alpha}$.
Recall that $V_{\ve \alpha}$ denotes the $(N-1)$---dimensional subspace of $\R^N$ orthogonal to the vector
${\ve \alpha}=(\alpha_1, \alpha_2,\ldots,\alpha_{N-1}, 1)$.
For any $M> 0$ one can choose a lattice $L_M\subset V_{\ve \alpha}$ of
determinant $1$ with
\be \mu(B^{N-1}_1\cap V_{\ve \alpha},L_M)>\frac{4M}{c({\ve \alpha})}\,.\label{Step_1}\ee
Since the inhomogeneous minima are independent of translation and since rational
lattices are dense in the space of all lattices, by Lemma
\ref{limit}, we may assume that $L_M\subset\Q^{N}$.
Applying Lemma \ref{main_lemma} to the lattice $L_M$ , we get a sequence
${\ve
a}(t)$, where by
(\ref{conditions_on_alpha}),
\bea 0<a_1(t)<a_2(t)<\ldots <a_{N}(t)\,\eea
for sufficiently large $t$.
%for $t$ large enough.
%

Observe that (\ref{asympt_alpha}) implies
(\ref{Density2}) with $a_i=a_i(t)$, $i=1,\ldots,N$, and $t$ large
enough. Next we show that, for sufficiently large $t$,  inequality
(\ref{Sharpness2}) also holds. To this end we define
the lattice $\Gamma_t$ by
\bea \Gamma_t=||{\ve a}(t)||^{-1/(N-1)}\Lambda_{{\ve a}(t)}.
\eea
By (\ref{asympt_aij}) and (\ref{asympt_h}),
the sequence of lattices
$L_t=\pi(\Gamma_t)$ converges to the lattice $L=\pi(L_M)$.
%.
%
Now put ${\ve \alpha}(t)=(a_1(t)/a_N(t),\ldots,a_{N-1}(t)/a_N(t),1)$.
The simplex $S_{{\ve \alpha}(t)}$ has the form
%
%\bea L_t=a_N(t)^{-1/(N-1)}L_{{\ve a}(t)}\,, \eea
%
%
\bea S_{{\ve \alpha}(t)}= \left\{(x_1,\ldots,x_{N})\in\R_{\ge 0}^N: \sum_{i=1}^{N-1}\frac{a_i(t)}{a_N(t)} x_i + x_N= 1\right\}\,.\eea
The point ${\ve p}=(1/(2(N-1)), \ldots, 1/(2(N-1)))$ is an inner point
of the simplex $S=\pi(S_{{\ve \alpha}}(1))$ and thus of all the simplicies $S_t=\pi(S_{{\ve
\alpha}(t)})$ for sufficiently large $t$. By (\ref{asympt_alpha}) and
Lemma \ref{limit}, the sequence $\mu(S_t-{\ve p},
L_t)$ converges to $\mu(S-{\ve p}, L)$.
Here we consider the sequence $\mu(S_t-{\ve p},
L_t)$ instead of $\mu(S_t, L_t)$ because the
distance functions of the family of star bodies in Lemma \ref{limit}
need to converge on the unit ball.
Now, since the inhomogeneous minima are independent of translation,
the sequence $\mu(S_t, L_t)$ converges to
$\mu(S, L)$. This clearly implies that the sequence $\mu(S_{{\ve
\alpha}(t)}, \Gamma_t)$ converges to $\mu(S_{{\ve \alpha}}(1), L_M)$.

Consequently, for all sufficiently large $t$ we have

\bea
\begin{split}
f_N({\ve a}(t)) & =\mu((a_N(t))^{-1}S_{{\ve \alpha}(t)},
||{\ve a}(t)||^{1/(N-1)}\Gamma_t)\\
& =||{\ve a}(t)||^{1/(N-1)}a_N(t)\mu(S_{{\ve
\alpha}(t)},\Gamma_t)
 \\ &>\frac{1}{2}||{\ve a}(t)||^{1/(N-1)}a_N(t)\mu(B^{N-1}_{R_{\ve \alpha}},L_M)
\\ & =\frac{c({\ve \alpha})}{2r_{\ve \alpha}(1)}||{\ve a}(t)||^{1/(N-1)}a_N(t)\mu(B^{N-1}_1,L_M)\,
 \\ &>2M\frac{||{\ve a}(t)||^{1/(N-1)}a_N(t)}{r_{\ve \alpha}(1)}
\\ & > M\frac{||{\ve a}(t)||^{1/(N-1)}a_N(t)}{r_{{\ve \alpha}(t)}(1)}\,
=M\frac{\sum_{i=1}^N ||{\ve a}(t)[i]||a_i(t)}{||{\ve a}(t)||^{1-1/(N-1)}}=M s({\ve a}(t))\,.
\end{split}
\eea
The theorem is proved.

\section{Acknowledgement}

The authors wish to thank Professor Anatoly Zhigljavsky for
valuable comments and discussions.

\end{document}